\documentclass[12pt]{amsart}
 
\usepackage{amssymb, amscd, txfonts}
\usepackage{graphicx}
\usepackage[all]{xy} 
 
\numberwithin{equation}{section}

\newtheorem{theorem}{Theorem}
\newtheorem{proposition}[theorem]{Proposition}
\newtheorem{lemma}[theorem]{Lemma}

\newtheorem{step}{Step}

\theoremstyle{definition}
\newtheorem{definition}[theorem]{Definition}

\theoremstyle{remark}

\renewcommand{\ker}{\operatorname{Ker}}

\newcommand{\Z}{\mathbb{Z}}
\newcommand{\Q}{\mathbb{Q}}
\newcommand{\R}{\mathbb{R}}
\newcommand{\C}{\mathbb{C}}

\newcommand{\G}{\Gamma}
\newcommand{\DG}{D/\Gamma}
\newcommand{\DU}{D/U(F)_{\mathbb{Z}}}
\newcommand{\UFZ}{U(F)_{\mathbb{Z}}}
\newcommand{\UFZZ}{U(F)_{\mathbb{Z}}^{+}}
\newcommand{\UFC}{U(F)_{\mathbb{C}}}
\newcommand{\UFcpt}{U(F)^{\Sigma_{F}}}
\newcommand{\GFZ}{\Gamma(F)_{\mathbb{Z}}}
\newcommand{\GFZbar}{\overline{\Gamma(F)}_{\mathbb{Z}}}
\newcommand{\GFRbar}{\overline{\Gamma(F)}_{\mathbb{R}}}
\newcommand{\DGcpt}{(D/\Gamma)^{\Sigma}}
\newcommand{\DUcpt}{(D/U(F)_{\mathbb{Z}})^{\Sigma_{F}}}

\DeclareMathOperator{\aut}{Aut}

\begin{document}

\title[]{Boundary branch divisor of toroidal compactifications}
\author[]{Shouhei Ma}
\thanks{Supported by KAKENHI 21H00971 and 20H00112.} 
\address{Department~of~Mathematics, Tokyo~Institute~of~Technology, Tokyo 152-8551, Japan}
\email{ma@math.titech.ac.jp}

\begin{abstract}
We prove that any toroidal compactification of arithmetic quotient of Hermitian symmetric domain 
has no boundary branch divisor, 
in the setting where the algebraic group is of adjoint type. 
\end{abstract} 

\maketitle

\section{Introduction}\label{sec: intro}

Toroidal compactification introduced by Ash-Mumford-Rapoport-Tai \cite{AMRT} is one of 
fundamental compactifications of arithmetic quotients of Hermitian symmetric domains 
that has various nice properties and has found many applications. 
The aim of this note is to prove absence of the phenomenon of ``irregular cusp'' 
in any toroidal compactification when the action of the Lie group is effective, 
a property that has been known for some examples 
but seems to have never been noticed in this generality. 
 
We work in the setting of \cite{AMRT} Chapter III.5. 
Thus let $D$ be a Hermitian symmetric domain and 
$\mathbb{G}$ be a connected semi-simple linear algebraic group over ${\Q}$ such that 
$\mathbb{G}(\mathbb{R})^{\circ} = {\aut}(D)^{\circ} =\colon G$ 
and $D=G/K$. 
Here $\mathbb{G}(\mathbb{R})$ is the group of real points of $\mathbb{G}$, 
$\circ$ in the superscript means identity component as Lie group, 
and $K$ is a maximal compact subgroup of $G$. 
In particular, $G$ acts on $D$ effectively in this setting. 
Thus, for example, we are considering as $G$ Lie groups like 
$G={\rm PSL}(2, {\R}), {\rm PSp}(2g, {\R}), {\rm PSO}(2, n)^{\circ}, {\rm PU}(p, q)$ etc. 
Let ${\G}$ be an arithmetic subgroup of $G$. 
(We do not assume that ${\G}$ is neat.) 

A toroidal compactification of ${\DG}$, say ${\DGcpt}$, 
is defined by choosing a ${\G}$-admissible collection of fans $\Sigma=(\Sigma_{F})$, 
one for each rational boundary component $F$ of $D$. 
If ${\UFZ}={\G}\cap U(F)$ is the integral part of the center of the unipotent part of the stabilizer of $F$, 
the fan $\Sigma_{F}$ defines a partial compactification 
${\DU}\hookrightarrow {\DUcpt}$ 
in the direction of $F$ by means of relative torus embedding. 
(See \S \ref{ssec: toroidal recall} for more detail.) 
The partial compactification ${\DUcpt}$ 
is accessible and suitable for various kinds of analysis, not just the geometry of ${\DUcpt}$ itself, 
but also Fourier-Jacobi expansion of modular forms, extension of period mapping, and so on. 
We have the fundamental open holomorphic map  
${\DUcpt}\to {\DGcpt}$, 
which can be thought of as providing a natural local chart 
around the boundary points of ${\DGcpt}$ lying over $F$, 
in the same sense as $D\to {\DG}$ providing a natural local chart 
for the interior points ${\DG}\subset {\DGcpt}$. 

We prove the following. 

\begin{theorem}\label{main}
The holomorphic map ${\DUcpt}\to {\DGcpt}$ 
has no boundary divisor as a ramification divisor. 
\end{theorem}

In some examples, this property has been observed individually by explicit calculation: 
for ${\rm PSp}(2g, {\R})$ in \cite{Ta} p.439, 
for\footnote{
Strictly speaking, ${\rm O}(2, n)$ is not connected as an algebraic group, 
but the construction of toroidal compactification still works.}
${\rm PO}^{+}(2, n)$ in \cite{Ko} Proposition 8.1 and \cite{GHS} Corollaries 2.22 and 2.29, 
and for ${\rm PU}(1, q)$ in \cite{Be} Corollary 4. 
We prove that this property always holds in full generality. 
The proof is intrinsic and uniform. 

To know (absence of) boundary branch divisor is 
a fundamental information in the study of the geometry of ${\DGcpt}$. 
For example, as the above list of papers implies, 
this is a necessary step in the construction of pluricanonical forms and more generally holomorphic tensors on ${\DGcpt}$. 
It is also necessary when calculating the vanishing order of ${\G}$-cusp forms 
along the boundary divisors of ${\DGcpt}$: 
by Theorem \ref{main}, this is equal to the vanishing order at the level of ${\DUcpt}$, 
which in turn is measured by the Fourier-Jacobi expansion. 
In particular, it is integral (as opposed to fractional). 

When ${\G}$ is neat, Theorem \ref{main} holds evidently because 
${\DUcpt}\to {\DGcpt}$ is unramified. 
For some applications there is no loss if one assumes ${\G}$ neat, 
but this is not always the case, 
especially when one wants to work in the direction of nonvanishing of some invariant of 
${\DG}$ itself, rather than of its cover. 

What is essential for Theorem \ref{main} to hold 
is the effectiveness of the $G$-action on $D$. 
In some situations, we have a natural cover $\pi \colon \tilde{G}\to G$, 
such as 
$\tilde{G} = {\rm SL}(2, {\R}), {\rm Sp}(2g, {\R}), {\rm SO}^{+}(2, n), {\rm SU}(p, q)$ etc, 
and the arithmetic group ${\G}$ in hand 
is defined as a subgroup of $\tilde{G}$. 
The kernel $G_{0}$ of $\tilde{G}\to G$ acts on $D$ trivially. 
There is one benefit in working with $\tilde{G}$, rather than with $G$: 
we have more weights of modular forms available. 
For example, odd weights for $\tilde{G} = {\rm SL}(2, {\R})$. 
In this situation of ${\G}\subset \tilde{G}$, 
if we work with $U(F)\subset \tilde{G}$, 
boundary branch divisor may arise. 
Classically this is known as \textit{irregular cusps} of modular curves, 
and such a phenomenon also happens in higher dimension (\cite{Ma}). 
Theorem \ref{main} does not mean absence of boundary branch divisor 
in this more general situation of ${\G}\subset \tilde{G}$. 
Instead, Theorem \ref{main} leads to a classification of 
the boundary branch divisors for ${\G}$ (Proposition \ref{prop: irregular}). 
It turns out that it is the finite abelian group $G_{0}/(G_{0}\cap {\G})$ that is essentially responsible 
for the presence of such divisors. 
This is apparent in the case of modular curves, 
and we find that this is still so in the general higher dimensional case. 
To understand what causes boundary ramification divisor 
for ${\G}\subset \tilde{G}$ is our motivation of this work. 
We explain this in \S \ref{sec: cover}.


\section{Proof}

In \S \ref{ssec: toroidal recall} we recall the construction of toroidal compactification 
following \cite{AMRT}. 
We will basically follow the notation in \cite{AMRT} Chapter III. 
In \S \ref{ssec: proof} we prove Theorem \ref{main}. 

\subsection{Toroidal compactification}\label{ssec: toroidal recall}

As in \S \ref{sec: intro}, let $D=G/K$ be a Hermitian symmetric domain 
where $G={\aut}(D)^{\circ}$, 
and $\mathbb{G}$ a connected semi-simple linear algebraic group over ${\Q}$ 
such that $\mathbb{G}({\R})^{\circ}=G$. 
Let $F$ be a rational boundary component of $D$. 
We denote by $N(F)$ the stabilizer of $F$ in $G$. 
Let $U(F)$ be the center of the unipotent radical of the identity component $N(F)^{\circ}$ of $N(F)$. 
Then $U(F)$ is an ${\R}$-linear space acted on by $N(F)$ by conjugation (\cite{AMRT} p.158). 
There is an open homogeneous cone $C(F)\subset U(F)$ acted on transitively by $N(F)$, 
which is self-adjoint with respect to certain natural metric on $U(F)$. 
If $F'\succ F$ is a rational boundary component containing $F$ in its closure, 
we have $U(F')\subset U(F)$. 
We write $C(F)^{\ast}\subset U(F)$ for the union of $C(F)$ and all $C(F')$ for such $F'$. 
For simplicity of exposition 
we also count $D$ as such a boundary component $F'$ for which $U(F')=\{ 0 \}$. 

We embed $D=G/K$ in its compact dual $\check{D}=G_{{\C}}/K_{{\C}}P_{-}$ and let 
$D(F)={\UFC}\cdot D \subset \check{D}$. 
We write $D(F)'=D(F)/{\UFC}$. 
Then there exists an $N(F)\cdot {\UFC}$-equivariant two-step holomorphic fibration 
\begin{equation*}
D(F) \to D(F)' \to F 
\end{equation*}
and an $N(F)\cdot {\UFC}$-equivariant map 
$\Phi \colon D(F) \to U(F)$ 
such that 
$D(F) \to D(F)'$ is a principal ${\UFC}$-bundle 
and $D=\Phi^{-1}(C(F))$ 
(\cite{AMRT} Chapter III.4.3 and p.158). 
Here, on the target $U(F)$ of $\Phi$, 
$N(F)$ acts by conjugation and 
$iU(F)\subset {\UFC}$ acts as translation by the imaginary part. 
This is the realization of $D$ as a Siegel domain of the third kind with respect to $F$. 

Let ${\G}$ be an arithmetic subgroup of $G$. 
We put 
${\GFZ}={\G}\cap N(F)$ and ${\UFZ}={\G}\cap U(F)$. 
We denote by $T(F)={\UFC}/{\UFZ}$ the algebraic torus associated to ${\UFZ}$. 
Then $D(F)/{\UFZ}$ is a principal $T(F)$-bundle over $D(F)'$. 
The map $\Phi$ descends to $D(F)/{\UFZ}\to U(F)$, 
still denoted by $\Phi$, 
and we have ${\DU}=\Phi^{-1}(C(F))$ in $D(F)/{\UFZ}$. 

Let $\Sigma=(\Sigma_{F})$ be a ${\G}$-admissible collection of fans 
in the sense of \cite{AMRT} Definition III.5.1. 
Each fan $\Sigma_{F}=(\sigma_{\alpha})$ is a rational polyhedral cone decomposition of 
$C(F)^{\ast}\subset U(F)$ preserved under the adjoint action of ${\GFZ}$. 
This defines a torus embedding 
$T(F)\hookrightarrow T(F)^{\Sigma_{F}}$. 
We take the relative torus embedding 
\begin{equation*}
(D(F)/{\UFZ})^{\Sigma_{F}} = (D(F)/{\UFZ}) \times_{T(F)} T(F)^{\Sigma_{F}} 
\end{equation*}
over $D(F)'$. 
Recall (\cite{AMRT} Chapter I.1) that the fan $\Sigma_{F}$ also defines 
a partial compactification ${\UFcpt}$ of $U(F)$ 
whose boundary is stratified into $U(F)$-orbits $U(F)/L(\sigma)$, $\sigma\in \Sigma_{F}$, 
where $L(\sigma)\subset U(F)$ is the linear span of $\sigma$.  
Then $\Phi$ extends continuously to  
\begin{equation*}
\Phi: (D(F)/{\UFZ})^{\Sigma_{F}} \to {\UFcpt}. 
\end{equation*}

Let ${\DUcpt}$ be the interior of the closure of ${\DU}$ in $(D(F)/{\UFZ})^{\Sigma_{F}}$. 
If $C(F)''\subset {\UFcpt}$ is the interior of the closure of $C(F)$ in ${\UFcpt}$, 
then (\cite{AMRT} p.161) 
\begin{equation*}
{\DUcpt} = \Phi^{-1}(C(F)'') \quad \subset \; \; (D(F)/{\UFZ})^{\Sigma_{F}}.  
\end{equation*}
The boundary of ${\DUcpt}$ is stratified into loci indexed by the cones $\sigma$ in $\Sigma_{F}$. 
When the relative interior of $\sigma$ is contained in $C(F)$, 
the boundary stratum $U(F)/L(\sigma)$ of ${\UFcpt}$ is contained in $C(F)''$, 
and the corresponding boundary stratum of ${\DUcpt}$ is $\Phi^{-1}(U(F)/L(\sigma))$. 
This is a principal bundle over $D(F)'$ 
for the quotient torus of $T(F)$ associated to the quotient lattice 
${\UFZ}/({\UFZ}\cap L(\sigma))$. 
The union of all such boundary strata of ${\DUcpt}$ 
is called the \textit{$F$-stratum} (\cite{AMRT} p.164).  

The toroidal compactification of ${\DG}$ associated to $\Sigma$ 
is defined by the gluing 
\begin{equation}\label{eqn: glue}
{\DGcpt} = ( \sqcup_{F} {\DUcpt})/\sim  
\end{equation}
where $F$ ranges over all rational boundary components of $D$ 
(including $D$ itself for which $U(F)$ is trivial), 
and $\sim$ is the equivalence relation generated by the following maps: 
\begin{itemize} 
\item Action of $\gamma \in {\G}$ giving 
${\DUcpt}\to (D/U(\gamma F)_{{\Z}})^{\Sigma_{\gamma F}}$. 
\item Gluing map $(D/U(F')_{{\Z}})^{\Sigma_{F'}}\to {\DUcpt}$ for $F\prec F'$. 
\end{itemize}
Then ${\DGcpt}$ is a compact Moishezon space containing ${\DG}$ as a Zariski open set 
and having a morphism to the Baily-Borel compactification of ${\DG}$. 
We have a natural open holomorphic map 
${\DUcpt}\to {\DGcpt}$. 
Points in the $F$-stratum of ${\DUcpt}$ are mapped to 
points in the image of $F$ in the Baily-Borel compactification.

\subsection{Proof of Theorem \ref{main}}\label{ssec: proof}

Now we prove Theorem \ref{main} in four steps. 
We take the quotient group 
${\GFZbar}={\GFZ}/{\UFZ}$, 
which makes sense because ${\UFZ}$ is normal in ${\GFZ}$. 
This group acts on ${\DUcpt}$ naturally. 
We begin with the following reduction. 

\begin{step}\label{step1}
For the proof of Theorem \ref{main} it suffices to prove the following: 

If an element $\gamma\in {\GFZbar}$ fixes a divisor in the boundary of ${\DUcpt}$ 
whose general point belongs to the $F$-stratum, then $\gamma={\rm id}$. 
\end{step}

\begin{proof}
First, for the proof of Theorem \ref{main}, 
it is sufficient to consider only boundary divisors whose general point belongs to the $F$-stratum. 
Indeed, a point in the boundary of ${\DUcpt}$ which does not belong to the $F$-stratum 
is contained in the image of the gluing map 
\begin{equation*}
p_{F'}\colon (D/U(F')_{{\Z}})^{\Sigma_{F'}}\to {\DUcpt} 
\end{equation*}
for some $F'\succ F$. 
The composition of ${\DUcpt}\to {\DGcpt}$ with $p_{F'}$ 
coincides with the natural map 
$(D/U(F')_{{\Z}})^{\Sigma_{F'}}\to {\DGcpt}$ for $F'$. 
Therefore the assertion for boundary divisors whose general point is not in the $F$-stratum 
follows from the corresponding assertion for some $F'\succ F$. 

The holomorphic map ${\DUcpt}\to {\DGcpt}$ factorizes as 
\begin{equation*}
{\DUcpt} \to {\DUcpt}/{\GFZbar} \to {\DGcpt}. 
\end{equation*}
By \cite{AMRT} p.175, the second map 
${\DUcpt}/{\GFZbar} \to {\DGcpt}$ is isomorphic in an open neighborhood of the $F$-stratum. 
This implies that Theorem \ref{main} follows if we could show that the first map 
${\DUcpt} \to {\DUcpt}/{\GFZbar}$ 
has no boundary divisor over $F$ as a ramification divisor. 
\end{proof}

Let $\gamma\in {\GFZbar}$ and $\Delta$ be an irreducible component of the boundary divisor of 
${\DUcpt}$ which is fixed by $\gamma$ and whose general point belongs to the $F$-stratum. 
We want to show that $\gamma={\rm id}$. 
Note first that $\gamma$ must be of finite order 
because the action of ${\GFZbar}$ on ${\DUcpt}$ is properly discontinuous 
(\cite{AMRT} Proposition III.6.10). 

\begin{step}\label{step2}
The adjoint action of $\gamma$ on $U(F)$ is trivial. 
\end{step}

\begin{proof}
The boundary divisor $\Delta$ corresponds to a ray 
$\sigma={\R}_{\geq 0} v$ in $\Sigma_{F}$ 
whose relative interior ${\R}_{>0} v$ is contained in $C(F)$. 
We choose as the generator $v$ a primitive vector of ${\UFZ}$. 
We shrink $\Delta$ to its Zariski open stratum so that 
$\Delta \to D(F)'$ is a principal bundle for the quotient torus of $T(F)$ 
associated to ${\UFZ}/{\Z}v$. 
Since we have a ${\GFZbar}$-equivariant map 
\begin{equation*}
\Phi : \Delta \twoheadrightarrow U(F)/{\R}v \; \; \subset \; {\UFcpt}, 
\end{equation*}
we see that the adjoint action of $\gamma$ on $U(F)$ 
preserves ${\R}v$ and acts on $U(F)/{\R}v$ trivially. 
Moreover, since $\gamma$ preserves the lattice ${\UFZ}$ and the cone $C(F)$, 
it must fix the unique primitive integral generator $v$ of $\sigma$. 
Thus $\gamma$ also acts on ${\R}v$ trivially. 
Since $\gamma$ is of finite order, we conclude that $\gamma$ acts on $U(F)$ trivially. 
\end{proof}

We can also prove Step \ref{step2} by looking at the $T(F)$-action on $\Delta$, 
instead of looking at $\Phi$. 
The next step is the key step.

\begin{step}\label{step3}
There exists an element $\alpha$ of $T(F)$ of finite order such that 
the $\gamma$-action on $D(F)/{\UFZ}$ coincides with 
the relative translation by $\alpha$ on every fiber of $\pi \colon D(F)/{\UFZ}\to D(F)'$.  
\end{step}

\begin{proof}
Since we have a ${\GFZbar}$-equivariant surjective map $\Delta\to D(F)'$, 
we find that $\gamma$ acts on $D(F)'$ trivially. 
Hence $\gamma$ preserves each fiber $\pi^{-1}(x)\simeq T(F)$ of $\pi$. 
Let $\gamma_{x}\colon \pi^{-1}(x) \to \pi^{-1}(x)$ be the action of $\gamma$ on $\pi^{-1}(x)$.  
By Step \ref{step2}, $\gamma$ commutes with every element of ${\UFC}$. 
Therefore the map $\gamma_{x}$ commutes with 
the translation by $T(F)$ on $\pi^{-1}(x)\simeq T(F)$. 
This shows that $\gamma_{x}$ itself is a translation, say by $\alpha(x)\in T(F)$. 
Since $\alpha(x)\in T(F)$ depends continuously on $x\in D(F)'$ 
and at the same time is of finite order at every $x$, 
we find that $\alpha(x)$ is constant for $x$. 
This proves our assertion. 
\end{proof}

\begin{step}\label{step4}
$\gamma={\rm id}$. 
\end{step}

\begin{proof}
We consider, as an auxiliary group, 
the normalizer of the lattice ${\UFZ}$ in $N(F)$ and denote it by $\Gamma(F)^{\ast}\subset N(F)$. 
By definition we have ${\UFZ}\lhd \Gamma(F)^{\ast}$. 
Note that $U(F)$ and ${\GFZ}$ are contained in $\Gamma(F)^{\ast}$. 
We put ${\GFRbar}=\Gamma(F)^{\ast}/{\UFZ}$. 
Then $U(F)/{\UFZ}$ and ${\GFZbar}$ are subgroups of ${\GFRbar}$, 
and we have 
\begin{equation}\label{eqn: trivial intersection}
(U(F)/{\UFZ}) \cap {\GFZbar} = \{ {\rm id} \} 
\end{equation}
in ${\GFRbar}$ 
by the definition ${\UFZ}=U(F)\cap {\GFZ}$ of ${\UFZ}$. 

The group ${\GFRbar}$ acts on $D(F)/{\UFZ}$. 
Since $N(F)$ acts on $D(F)$ effectively by our initial setting, 
and since ${\UFZ}$ is discrete and acts on $D(F)$ freely, 
we see that the action of ${\GFRbar}$ on $D(F)/{\UFZ}$ is also effective. 
By Step \ref{step3}, we know that the elements  
$\gamma\in {\GFZbar}$ and $\alpha \in U(F)/{\UFZ}$ 
have the same action on $D(F)/{\UFZ}$. 
Therefore we find that $\gamma = \alpha$ as elements of ${\GFRbar}$ 
by the effectivity of the action. 
By \eqref{eqn: trivial intersection}, we conclude that $\gamma={\rm id}$. 
This completes the proof of Theorem \ref{main}. 
\end{proof}

\section{Covering group}\label{sec: cover}

In this section we consider the more general case where 
${\G}$ is defined as a subgroup of a covering group of ${\aut}(D)^{\circ}$. 
Theorem \ref{main} leads to a classification of 
the boundary branch divisors of ${\DGcpt}$ in this situation 
(Proposition \ref{prop: irregular}).

\subsection{Toroidal compactification with sublattices}\label{ssec: sublattice}

In this subsection, which is the preliminary for the next \S \ref{ssec: ineffective kernel}, 
we work in the setting of \S \ref{ssec: toroidal recall}. 
Thus ${\G}$ is an arithmetic subgroup of $G={\aut}(D)^{\circ}$. 
Suppose that a finite-index sublattice ${\UFZZ}\subset {\UFZ}$ is given 
for each rational boundary component $F$, satisfying the following conditions: 
\begin{itemize}
\item $U(F')_{{\Z}}^{+} = U(F')\cap {\UFZZ}$ when $F'\succ F$. 
\item $U(\gamma F)_{{\Z}}^{+} = \gamma {\UFZZ} \gamma^{-1}$ for $\gamma \in {\G}$. 
In particular, ${\UFZZ}$ is preserved under the adjoint action of ${\GFZ}$ on ${\UFZ}$. 
\end{itemize}
Let $A_{F}={\UFZ}/{\UFZZ}$ and $T(F)_{+}={\UFC}/{\UFZZ}$. 
Since the fan $\Sigma_{F}$ is also rational with respect to ${\UFZZ}$, 
it defines a torus embedding 
$T(F)_{+}^{\Sigma_{F}}$ of $T(F)_{+}$, 
and we have 
$T(F)^{\Sigma_{F}}=T(F)_{+}^{\Sigma_{F}}/A_{F}$. 
This defines the partial compactification 
$(D/{\UFZZ})^{\Sigma_{F}}$ of $D/{\UFZZ}$. 
Working with ${\UFZZ}$ in place of ${\UFZ}$, 
we can define the toroidal compactification of ${\DG}$ with $({\UFZZ})_{F}$ 
by the same gluing procedure as \eqref{eqn: glue}, 
covered by the partial compactifications $(D/{\UFZZ})^{\Sigma_{F}}$. 
The result is still naturally isomorphic to the usual ${\DGcpt}$ with $({\UFZ})_{F}$, 
as can be seen from the fact that 
the natural map between their local models around the $F$-stratum  
\begin{equation*}
(D/{\UFZZ})^{\Sigma_{F}}/({\GFZ}/{\UFZZ}) \to
{\DUcpt}/{\GFZbar}  
\end{equation*}
is isomorphic. 
Here note that 
$(D/{\UFZZ})^{\Sigma_{F}}/A_{F} = {\DUcpt}$. 

\begin{lemma}\label{lem: sublattice}
Let $\Delta$ be a boundary divisor of $(D/{\UFZZ})^{\Sigma_{F}}$ 
corresponding to a ray $\sigma$ of $\Sigma_{F}$. 
The map  
$(D/{\UFZZ})^{\Sigma_{F}} \to {\DGcpt}$ 
is ramified along $\Delta$ if and only if 
$\sigma \cap {\UFZ} \ne \sigma \cap {\UFZZ}$, 
with the ramification caused by the cyclic subgroup 
$({\R}\sigma \cap {\UFZ})/({\R}\sigma \cap {\UFZZ})$ 
of $A_{F}$. 
\end{lemma}

\begin{proof}
Since $(D/{\UFZZ})^{\Sigma_{F}} \to {\DGcpt}$ factorizes as 
\begin{equation*}
(D/{\UFZZ})^{\Sigma_{F}} \to {\DUcpt} \to {\DGcpt}, 
\end{equation*}
Theorem \ref{main} shows that 
the boundary ramification divisors of the map $(D/{\UFZZ})^{\Sigma_{F}} \to {\DGcpt}$ 
are the same as those of 
$(D/{\UFZZ})^{\Sigma_{F}} \to {\DUcpt}$,  
which in turn correspond to those of 
$T(F)_{+}^{\Sigma_{F}}\to T(F)^{\Sigma_{F}}$.  
\end{proof}

\subsection{Effect of ineffective kernel}\label{ssec: ineffective kernel}

Let $G=\mathbb{G}({\R})^{\circ}={\aut}(D)^{\circ}$ 
be as in \S \ref{ssec: toroidal recall}. 
Let $\pi \colon \tilde{G}\to G$ be a finite covering of Lie groups. 
$\tilde{G}$ is a connected Lie group, not assumed to be associated to an algebraic group over ${\Q}$. 
The kernel $G_{0}={\ker}(\pi)$ of $\pi$ is the center of $\tilde{G}$, 
and $\tilde{G}$ acts on $D$ with ineffective kernel $G_{0}$. 
For $U(F)\subset G$ we denote by 
$U(F)^{+}\subset \tilde{G}$ the identity component of $\pi^{-1}(U(F))$. 
Then $U(F)^{+}\to U(F)$ is isomorphic and 
$\pi^{-1}(U(F)) = U(F)^{+}\times G_{0}$. 
$U(F)^{+}$ is what we should write ``$U(F)$'' for $\tilde{G}$. 

Let now ${\G}$ be an arithmetic subgroup of $\tilde{G}$, 
namely $\pi({\G})\subset G$ arithmetic. 
We write 
${\UFZ}={\G}\cap U(F)^{+}$ and 
$U(F)_{{\Z}}'={\G}\cap \pi^{-1}(U(F))$. 
We have 
\begin{equation*}
{\UFZ} \: \subset \: {\UFZ}\times (G_{0}\cap {\G}) \: \subset \: U(F)_{{\Z}}'. 
\end{equation*}
Note that $\pi(U(F)_{{\Z}}')\subset G$ is what we previously wrote ``${\UFZ}$'' for $\pi({\G})\subset G$. 
(We want to stick to the notation ${\UFZ}$ for the lattice of unipotent elements 
in the given arithmetic group ${\G}$.) 
It is necessary to work with ${\G}\subset \tilde{G}$, 
rather than with $\pi({\G})\subset G$, 
when considering modular forms of (say) odd weight. 
We must work with ${\UFZ}\subset U(F)^{+}$, 
rather than with $U(F)_{{\Z}}'$, 
when considering Fourier-Jacobi expansion of such modular forms. 

Let $\Sigma=(\Sigma_{F})$ be a ${\G}$-admissible collection of fans,  
where each $\Sigma_{F}$ is a fan in $U(F)^{+}\simeq U(F)$. 
It is possible to construct 
the toroidal compactification ${\DGcpt}$ of ${\DG}$  
with the lattices ${\UFZ}\subset U(F)^{+}$. 
This is the same as the toroidal compactification of $D/\pi({\G})$ with 
the sublattices $(\pi({\UFZ}))_{F}$ of $(\pi(U(F)_{{\Z}}'))_{F}$, 
hence with $(\pi(U(F)_{{\Z}}'))_{F}$, 
which satisfies the conditions in \S \ref{ssec: sublattice}. 

\begin{proposition}\label{prop: irregular}
For ${\G}\subset \tilde{G}$ the natural map 
$(D/{\UFZ})^{\Sigma_{F}}\to {\DGcpt}$ 
is ramified along a boundary divisor 
corresponding to a ray $\sigma\in \Sigma_{F}$ if and only if 
$\sigma\cap \pi({\UFZ}) \ne \sigma\cap \pi(U(F)_{{\Z}}')$, 
with the ramification caused by the cyclic subgroup 
$({\R}\sigma\cap \pi(U(F)_{{\Z}}'))/({\R}\sigma\cap \pi({\UFZ}))$ 
of $G_{0}/(G_{0}\cap {\G})$. 
\end{proposition}

\begin{proof}
This follows from Lemma \ref{lem: sublattice}. 
Note that 
\begin{equation*}
\pi(U(F)_{{\Z}}')/\pi({\UFZ}) \simeq 
U(F)_{{\Z}}'/ ({\UFZ}\times (G_{0}\cap {\G})) 
\end{equation*}
is naturally a subgroup of $G_{0}/(G_{0}\cap {\G})$. 
\end{proof}

The analogy with the case of modular curves would justify the following terminology. 

\begin{definition}
We call a cusp $F$ \textit{irregular} when 
$\pi(U(F)_{{\Z}}') \ne \pi({\UFZ})$, or equivalently, 
$U(F)_{{\Z}}' \ne {\UFZ}\times (G_{0}\cap {\G})$. 
\end{definition}

By Proposition \ref{prop: irregular}, 
boundary branch divisor may arise over $F$ only when $F$ is irregular. 
It is evident that, if $F$ is regular, any $F'\succ F$ is also regular. 
It is also clear that, if ${\G}$ is neat (with $\tilde{G}$ associated to an algebraic group over ${\Q}$), 
every cusp $F$ is regular because $U(F)_{{\Z}}'={\UFZ}$.



\begin{thebibliography}{99}

\bibitem{AMRT}Ash, A.; Mumford, D.; Rapoport, M.; Tai, Y. 
\textit{Smooth compactifications of locally symmetric varieties.} 2nd edition. 
Cambridge Univ. Press, 2010. 

\bibitem{Be}Behrens, N. 
\textit{Singularities of ball quotients.} 
Geom. Dedicata \textbf{159} (2012), 389--407. 

\bibitem{GHS}Gritsenko, V.; Hulek, K.; Sankaran, G. 
\textit{The Kodaira dimension of the moduli of $K3$ surfaces.} 
Invent. Math. \textbf{169} (2007), 519--567. 

\bibitem{Ko}Kond\=o, S. 
\textit{On the Kodaira dimension of the moduli space of $K3$ surfaces.} 
Compositio Math. \textbf{89} (1993), 251--299. 

\bibitem{Ma}Ma, S. 
\textit{Irregular cusps of orthogonal modular varieties.} 
arXiv:2101.02950


\bibitem{Ta}Tai, Y. 
\textit{On the Kodaira dimension of the moduli space of Abelian varieties.} 
Invent. Math. \textbf{68} (1982), 425--439. 


\end{thebibliography}
\end{document}